\documentclass[10pt]{amsart}

\usepackage{fontenc}

\makeatletter
 
\def\diagram{\m@th\leftwidth=\z@ \rightwidth=\z@ \topheight=\z@
\botheight=\z@ \setbox\@picbox\hbox\bgroup}
 
\def\enddiagram{\egroup\wd\@picbox\rightwidth\unitlength
\ht\@picbox\topheight\unitlength \dp\@picbox\botheight\unitlength
\hskip\leftwidth\unitlength\box\@picbox}
 
\def\bfig{\begin{diagram}}
\def\efig{\end{diagram}}
\newcount\wideness \newcount\leftwidth \newcount\rightwidth
\newcount\highness \newcount\topheight \newcount\botheight
 
\def\ratchet#1#2{\ifnum#1<#2 \global #1=#2 \fi}
 
\def\putbox(#1,#2)#3{%
\horsize{\wideness}{#3} \divide\wideness by 2
{\advance\wideness by #1 \ratchet{\rightwidth}{\wideness}}
{\advance\wideness by -#1 \ratchet{\leftwidth}{\wideness}}
\vertsize{\highness}{#3} \divide\highness by 2
{\advance\highness by #2 \ratchet{\topheight}{\highness}}
{\advance\highness by -#2 \ratchet{\botheight}{\highness}}
\put(#1,#2){\makebox(0,0){$#3$}}}
 
\def\putlbox(#1,#2)#3{%
\horsize{\wideness}{#3}
{\advance\wideness by #1 \ratchet{\rightwidth}{\wideness}}
{\ratchet{\leftwidth}{-#1}}
\vertsize{\highness}{#3} \divide\highness by 2
{\advance\highness by #2 \ratchet{\topheight}{\highness}}
{\advance\highness by -#2 \ratchet{\botheight}{\highness}}
\put(#1,#2){\makebox(0,0)[l]{$#3$}}}
 
\def\putrbox(#1,#2)#3{%
\horsize{\wideness}{#3}
{\ratchet{\rightwidth}{#1}}
{\advance\wideness by -#1 \ratchet{\leftwidth}{\wideness}}
\vertsize{\highness}{#3} \divide\highness by 2
{\advance\highness by #2 \ratchet{\topheight}{\highness}}
{\advance\highness by -#2 \ratchet{\botheight}{\highness}}
\put(#1,#2){\makebox(0,0)[r]{$#3$}}}

\def\adjust[#1]{} 
 
\newcount \coefa
\newcount \coefb
\newcount \coefc
\newcount\tempcounta
\newcount\tempcountb
\newcount\tempcountc
\newcount\tempcountd
\newcount\xext
\newcount\yext
\newcount\xoff
\newcount\yoff
\newcount\gap%
\newcount\arrowtypea
\newcount\arrowtypeb
\newcount\arrowtypec
\newcount\arrowtyped
\newcount\arrowtypee
\newcount\height
\newcount\width
\newcount\xpos
\newcount\ypos
\newcount\run
\newcount\rise
\newcount\arrowlength
\newcount\halflength
\newcount\arrowtype
\newdimen\tempdimen
\newdimen\xlen
\newdimen\ylen
\newsavebox{\tempboxa}%
\newsavebox{\tempboxb}%
\newsavebox{\tempboxc}%
 
\newdimen\w@dth
 
\def\setw@dth#1#2{\setbox\z@\hbox{\m@th$#1$}\w@dth=\wd\z@
\setbox\@ne\hbox{\m@th$#2$}\ifnum\w@dth<\wd\@ne \w@dth=\wd\@ne \fi
\advance\w@dth by 1.2em}
 
\def\t@^#1_#2{\allowbreak\def\n@one{#1}\def\n@two{#2}\mathrel
{\setw@dth{#1}{#2}
\mathop{\hbox to \w@dth{\rightarrowfill}}\limits
\ifx\n@one\empty\else ^{\box\z@}\fi
\ifx\n@two\empty\else _{\box\@ne}\fi}}
\def\t@@^#1{\@ifnextchar_{\t@^{#1}}{\t@^{#1}_{}}}
\def\to{\@ifnextchar^{\t@@}{\t@@^{}}}
 
\def\t@left^#1_#2{\def\n@one{#1}\def\n@two{#2}\mathrel{\setw@dth{#1}{#2}
\mathop{\hbox to \w@dth{\leftarrowfill}}\limits
\ifx\n@one\empty\else ^{\box\z@}\fi
\ifx\n@two\empty\else _{\box\@ne}\fi}}
\def\t@@left^#1{\@ifnextchar_{\t@left^{#1}}{\t@left^{#1}_{}}}
\def\toleft{\@ifnextchar^{\t@@left}{\t@@left^{}}}
 
\def\two@^#1_#2{\allowbreak
\def\n@one{#1}\def\n@two{#2}\mathrel{\setw@dth{#1}{#2}
\mathop{\vcenter{\lineskip\z@\baselineskip\z@
                 \hbox to \w@dth{\rightarrowfill}%
                 \hbox to \w@dth{\rightarrowfill}}%
       }\limits
\ifx\n@one\empty\else ^{\box\z@}\fi
\ifx\n@two\empty\else _{\box\@ne}\fi}}
\def\tw@@^#1{\@ifnextchar _{\two@^{#1}}{\two@^{#1}_{}}}
\def\two{\@ifnextchar ^{\tw@@}{\tw@@^{}}}
 
\def\tofr@^#1_#2{\def\n@one{#1}\def\n@two{#2}\mathrel{\setw@dth{#1}{#2}
\mathop{\vcenter{\hbox to \w@dth{\rightarrowfill}\kern-1.7ex
                 \hbox to \w@dth{\leftarrowfill}}%
       }\limits
\ifx\n@one\empty\else ^{\box\z@}\fi
\ifx\n@two\empty\else _{\box\@ne}\fi}}
\def\t@fr@^#1{\@ifnextchar_ {\tofr@^{#1}}{\tofr@^{#1}_{}}}
\def\tofro{\@ifnextchar^ {\t@fr@}{\t@fr@^{}}}

\def\mon{\mathop{\m@th\hbox to
      14.6\P@{\lasyb\char'51\hskip-2.1\P@$\arrext$\hss
$\mathord\rightarrow$}}\limits} 
\def\leftmono{\mathrel{\m@th\hbox to
14.6\P@{$\mathord\leftarrow$\hss$\arrext$\hskip-2.1\P@\lasyb\char'50%
}}\limits} 
\mathchardef\arrext="0200       

\def\settypes(#1,#2,#3){\arrowtypea#1 \arrowtypeb#2 \arrowtypec#3}
\def\settoheight#1#2{\setbox\@tempboxa\hbox{#2}#1\ht\@tempboxa\relax}%
\def\settodepth#1#2{\setbox\@tempboxa\hbox{#2}#1\dp\@tempboxa\relax}%
\def\settokens`#1`#2`#3`#4`{%
     \def\tokena{#1}\def\tokenb{#2}\def\tokenc{#3}\def\tokend{#4}}
\def\setsqparms[#1`#2`#3`#4;#5`#6]{%
\arrowtypea #1
\arrowtypeb #2
\arrowtypec #3
\arrowtyped #4
\width #5
\height #6
}
\def\setpos(#1,#2){\xpos=#1 \ypos#2}

\def\settriparms[#1`#2`#3;#4]{\settripairparms[#1`#2`#3`1`1;#4]}%
 
\def\settripairparms[#1`#2`#3`#4`#5;#6]{%
\arrowtypea #1
\arrowtypeb #2
\arrowtypec #3
\arrowtyped #4
\arrowtypee #5
\width #6
\height #6
}
 
\def\resetparms{\settripairparms[1`1`1`1`1;500]\width 500}
 
\resetparms
 
\def\mvector(#1,#2)#3{
\put(0,0){\vector(#1,#2){#3}}%
\put(0,0){\vector(#1,#2){26}}%
}
\def\evector(#1,#2)#3{{
\arrowlength #3
\put(0,0){\vector(#1,#2){\arrowlength}}%
\advance \arrowlength by-30
\put(0,0){\vector(#1,#2){\arrowlength}}%
}}
 
\def\horsize#1#2{%
\settowidth{\tempdimen}{$#2$}%
#1=\tempdimen
\divide #1 by\unitlength
}
 
\def\vertsize#1#2{%
\settoheight{\tempdimen}{$#2$}%
#1=\tempdimen
\settodepth{\tempdimen}{$#2$}%
\advance #1 by\tempdimen
\divide #1 by\unitlength
}
 
\def\putvector(#1,#2)(#3,#4)#5#6{{%
\ifnum3<\arrowtype
\putdashvector(#1,#2)(#3,#4)#5\arrowtype
\else
\ifnum\arrowtype<-3
\putdashvector(#1,#2)(#3,#4)#5\arrowtype
\else
\xpos=#1
\ypos=#2
\run=#3
\rise=#4
\arrowlength=#5
\ifnum \arrowtype<0
    \ifnum \run=0
        \advance \ypos by-\arrowlength
    \else
        \tempcounta \arrowlength
        \multiply \tempcounta by\rise
        \divide \tempcounta by\run
        \ifnum\run>0
            \advance \xpos by\arrowlength
            \advance \ypos by\tempcounta
        \else
            \advance \xpos by-\arrowlength
            \advance \ypos by-\tempcounta
        \fi
    \fi
    \multiply \arrowtype by-1
    \multiply \rise by-1
    \multiply \run by-1
\fi
\ifcase \arrowtype
\or \put(\xpos,\ypos){\vector(\run,\rise){\arrowlength}}%
\or \put(\xpos,\ypos){\mvector(\run,\rise)\arrowlength}%
\or \put(\xpos,\ypos){\evector(\run,\rise){\arrowlength}}%
\fi\fi\fi
}}
 
\def\putsplitvector(#1,#2)#3#4{
\xpos #1
\ypos #2
\arrowtype #4
\halflength #3
\arrowlength #3
\gap 140
\advance \halflength by-\gap
\divide \halflength by2
\ifnum\arrowtype>0
   \ifcase \arrowtype
   \or \put(\xpos,\ypos){\line(0,-1){\halflength}}%
       \advance\ypos by-\halflength
       \advance\ypos by-\gap
       \put(\xpos,\ypos){\vector(0,-1){\halflength}}%
   \or \put(\xpos,\ypos){\line(0,-1)\halflength}%
       \put(\xpos,\ypos){\vector(0,-1)3}%
       \advance\ypos by-\halflength
       \advance\ypos by-\gap
       \put(\xpos,\ypos){\vector(0,-1){\halflength}}%
   \or \put(\xpos,\ypos){\line(0,-1)\halflength}%
       \advance\ypos by-\halflength
       \advance\ypos by-\gap
       \put(\xpos,\ypos){\evector(0,-1){\halflength}}%
   \fi
\else \arrowtype=-\arrowtype
   \ifcase\arrowtype
   \or \advance \ypos by-\arrowlength
       \put(\xpos,\ypos){\line(0,1){\halflength}}%
       \advance\ypos by\halflength
       \advance\ypos by\gap
       \put(\xpos,\ypos){\vector(0,1){\halflength}}%
   \or \advance \ypos by-\arrowlength
       \put(\xpos,\ypos){\line(0,1)\halflength}%
       \put(\xpos,\ypos){\vector(0,1)3}%
       \advance\ypos by\halflength
       \advance\ypos by\gap
       \put(\xpos,\ypos){\vector(0,1){\halflength}}%
   \or \advance \ypos by-\arrowlength
       \put(\xpos,\ypos){\line(0,1)\halflength}%
       \advance\ypos by\halflength
       \advance\ypos by\gap
       \put(\xpos,\ypos){\evector(0,1){\halflength}}%
   \fi
\fi
}
 
\def\putmorphism(#1)(#2,#3)[#4`#5`#6]#7#8#9{{%
\run #2
\rise #3
\ifnum\rise=0
  \puthmorphism(#1)[#4`#5`#6]{#7}{#8}#9%
\else\ifnum\run=0
  \putvmorphism(#1)[#4`#5`#6]{#7}{#8}#9%
\else
\setpos(#1)%
\arrowlength #7
\arrowtype #8
\ifnum\run=0
\else\ifnum\rise=0
\else
\ifnum\run>0
    \coefa=1
\else
   \coefa=-1
\fi
\ifnum\arrowtype>0
   \coefb=0
   \coefc=-1
\else
   \coefb=\coefa
   \coefc=1
   \arrowtype=-\arrowtype
\fi
\width=2
\multiply \width by\run
\divide \width by\rise
\ifnum \width<0  \width=-\width\fi
\advance\width by60
\if l#9 \width=-\width\fi
\putbox(\xpos,\ypos){#4}
{\multiply \coefa by\arrowlength
\advance\xpos by\coefa
\multiply \coefa by\rise
\divide \coefa by\run
\advance \ypos by\coefa
\putbox(\xpos,\ypos){#5} }%
{\multiply \coefa by\arrowlength
\divide \coefa by2
\advance \xpos by\coefa
\advance \xpos by\width
\multiply \coefa by\rise
\divide \coefa by\run
\advance \ypos by\coefa
\if l#9%
   \putrbox(\xpos,\ypos){#6}%
\else\if r#9%
   \putlbox(\xpos,\ypos){#6}%
\fi\fi }%
{\multiply \rise by-\coefc
\multiply \run by-\coefc
\multiply \coefb by\arrowlength
\advance \xpos by\coefb
\multiply \coefb by\rise
\divide \coefb by\run
\advance \ypos by\coefb
\multiply \coefc by70
\advance \ypos by\coefc
\multiply \coefc by\run
\divide \coefc by\rise
\advance \xpos by\coefc
\multiply \coefa by140
\multiply \coefa by\run
\divide \coefa by\rise
\advance \arrowlength by\coefa
\ifcase\arrowtype
\or \put(\xpos,\ypos){\vector(\run,\rise){\arrowlength}}%
\or \put(\xpos,\ypos){\mvector(\run,\rise){\arrowlength}}%
\or \put(\xpos,\ypos){\evector(\run,\rise){\arrowlength}}%
\fi}\fi\fi\fi\fi}}

\newcount\numbdashes \newcount\lengthdash \newcount\increment
 
\def\howmanydashes{
\numbdashes=\arrowlength \lengthdash=40
\divide\numbdashes by \lengthdash
\lengthdash=\arrowlength
\divide\lengthdash by \numbdashes
\increment=\lengthdash
\multiply\lengthdash by 3
\divide\lengthdash by 5
}
 
\def\putdashvector(#1)(#2,#3)#4#5{%
\ifnum#3=0 \putdashhvector(#1){#4}#5
\else
\ifnum#2=0
\putdashvvector(#1){#4}#5\fi\fi}
 
\def\putdashhvector(#1,#2)#3#4{{%
\arrowlength=#3 \howmanydashes
\multiput(#1,#2)(\increment,0){\numbdashes}%
{\vrule height .4pt width \lengthdash\unitlength}
\arrowtype=#4 \xpos=#1
\ifnum\arrowtype<0 \advance\arrowtype by 7 \fi
\ifcase\arrowtype
\or \advance\xpos by 10
    \put(\xpos,#2){\vector(-1,0){\lengthdash}}
    \advance\xpos by 40
    \put(\xpos,#2){\vector(-1,0){\lengthdash}}
\or \advance \xpos by 10
    \put(\xpos,#2){\vector(-1,0){\lengthdash}}
    \advance\xpos by  \arrowlength
    \advance\xpos by  -50
    \put(\xpos,#2){\vector(-1,0){\lengthdash}}
\or \advance\xpos by 10
    \put(\xpos,#2){\vector(-1,0){\lengthdash}}
\or \advance\xpos by \arrowlength
    \advance\xpos by -\lengthdash
    \put(\xpos,#2){\vector(1,0){\lengthdash}}
\or {\advance\xpos by 10
    \put(\xpos,#2){\vector(1,0){\lengthdash}}}
    \advance\xpos by \arrowlength
    \advance\xpos by -\lengthdash
    \put(\xpos,#2){\vector(1,0){\lengthdash}}
\or \advance\xpos by \arrowlength
    \advance\xpos by -\lengthdash
    \put(\xpos,#2){\vector(1,0){\lengthdash}}
    \advance\xpos by -40
    \put(\xpos,#2){\vector(1,0){\lengthdash}}
   \fi
}}
 
\def\putdashvvector(#1,#2)#3#4{{%
\arrowlength=#3 \howmanydashes
\ypos=#2 \advance\ypos by -\arrowlength
\multiput(#1,#2)(0,\increment){\numbdashes}%
    {\vrule width .4pt height \lengthdash\unitlength}
\arrowtype=#4 \ypos=#2
\ifnum\arrowtype<0 \advance\arrowtype by 7 \fi
\ifcase\arrowtype
\or \advance\ypos by \arrowlength \advance\ypos by -40
    \put(#1,\ypos){\vector(0,1){\lengthdash}}
    \advance\ypos by -40
    \put(#1,\ypos){\vector(0,1){\lengthdash}}
\or \advance\ypos by 10
    \put(#1,\ypos){\vector(0,1){\lengthdash}}
    \advance\ypos by \arrowlength \advance\ypos by -40
    \put(#1,\ypos){\vector(0,1){\lengthdash}}
\or \advance\ypos by \arrowlength \advance\ypos by -40
    \put(#1,\ypos){\vector(0,1){\lengthdash}}
\or \advance\ypos by 10
    \put(#1,\ypos){\vector(0,-1){\lengthdash}}
\or \advance\ypos by 10
    \put(#1,\ypos){\vector(0,-1){\lengthdash}}
    \advance\ypos by \arrowlength \advance\ypos by -40
    \put(#1,\ypos){\vector(0,-1){\lengthdash}}
\or \advance\ypos by 10
    \put(#1,\ypos){\vector(0,-1){\lengthdash}}
    \advance\ypos by 40
    \put(#1,\ypos){\vector(0,-1){\lengthdash}}
\fi
}}
 
\def\puthmorphism(#1,#2)[#3`#4`#5]#6#7#8{{%
\xpos #1
\ypos #2
\width #6
\arrowlength #6
\arrowtype=#7
\putbox(\xpos,\ypos){#3\vphantom{#4}}%
{\advance \xpos by\arrowlength
\putbox(\xpos,\ypos){\vphantom{#3}#4}}%
\horsize{\tempcounta}{#3}%
\horsize{\tempcountb}{#4}%
\divide \tempcounta by2
\divide \tempcountb by2
\advance \tempcounta by30
\advance \tempcountb by30
\advance \xpos by\tempcounta
\advance \arrowlength by-\tempcounta
\advance \arrowlength by-\tempcountb
\putvector(\xpos,\ypos)(1,0)\arrowlength\arrowtype
\divide \arrowlength by2
\advance \xpos by\arrowlength
\vertsize{\tempcounta}{#5}%
\divide\tempcounta by2
\advance \tempcounta by20
\if a#8 %
   \advance \ypos by\tempcounta
   \putbox(\xpos,\ypos){#5}%
\else
   \advance \ypos by-\tempcounta
   \putbox(\xpos,\ypos){#5}%
\fi}}
 
\def\putvmorphism(#1,#2)[#3`#4`#5]#6#7#8{{%
\xpos #1
\ypos #2
\arrowlength #6
\arrowtype #7
\settowidth{\xlen}{$#5$}%
\putbox(\xpos,\ypos){#3}%
{\advance \ypos by-\arrowlength
\putbox(\xpos,\ypos){#4}}%
{\advance\arrowlength by-140
\advance \ypos by-70
\ifdim\xlen>0pt
   \if m#8%
      \putsplitvector(\xpos,\ypos)\arrowlength\arrowtype
   \else
   \putvector(\xpos,\ypos)(0,-1)\arrowlength\arrowtype
   \fi
\else
   \putvector(\xpos,\ypos)(0,-1)\arrowlength\arrowtype
\fi}%
\ifdim\xlen>0pt
   \divide \arrowlength by2
   \advance\ypos by-\arrowlength
   \if l#8%
      \advance \xpos by-40
      \putrbox(\xpos,\ypos){#5}%
   \else\if r#8%
      \advance \xpos by40
      \putlbox(\xpos,\ypos){#5}%
   \else
      \putbox(\xpos,\ypos){#5}%
   \fi\fi
\fi
}}
 
\def\putsquarep<#1>(#2)[#3;#4`#5`#6`#7]{{%
\setsqparms[#1]%
\setpos(#2)%
\settokens`#3`%
\puthmorphism(\xpos,\ypos)[\tokenc`\tokend`{#7}]{\width}{\arrowtyped}b%
\advance\ypos by \height
\puthmorphism(\xpos,\ypos)[\tokena`\tokenb`{#4}]{\width}{\arrowtypea}a%
\putvmorphism(\xpos,\ypos)[``{#5}]{\height}{\arrowtypeb}l%
\advance\xpos by \width
\putvmorphism(\xpos,\ypos)[``{#6}]{\height}{\arrowtypec}r%
}}
 
\def\putsquare{\@ifnextchar <{\putsquarep}{\putsquarep%
   <\arrowtypea`\arrowtypeb`\arrowtypec`\arrowtyped;\width`\height>}}
\def\square{\@ifnextchar< {\squarep}{\squarep
   <\arrowtypea`\arrowtypeb`\arrowtypec`\arrowtyped;\width`\height>}}
\def\squarep<#1>[#2`#3`#4`#5;#6`#7`#8`#9]{{
\setsqparms[#1]
\diagram
\putsquarep<\arrowtypea`\arrowtypeb`\arrowtypec`
\arrowtyped;\width`\height>
(0,0)[#2`#3`#4`{#5};#6`#7`#8`{#9}]
\enddiagram
}}                                                 
\def\putptrianglep<#1>(#2,#3)[#4`#5`#6;#7`#8`#9]{{%
\settriparms[#1]%
\xpos=#2 \ypos=#3
\advance\ypos by \height
\puthmorphism(\xpos,\ypos)[#4`#5`{#7}]{\height}{\arrowtypea}a%
\putvmorphism(\xpos,\ypos)[`#6`{#8}]{\height}{\arrowtypeb}l%
\advance\xpos by\height
\putmorphism(\xpos,\ypos)(-1,-1)[``{#9}]{\height}{\arrowtypec}r%
}}
 
\def\putptriangle{\@ifnextchar <{\putptrianglep}{\putptrianglep
   <\arrowtypea`\arrowtypeb`\arrowtypec;\height>}}
\def\ptriangle{\@ifnextchar <{\ptrianglep}{\ptrianglep
   <\arrowtypea`\arrowtypeb`\arrowtypec;\height>}}
\def\ptrianglep<#1>[#2`#3`#4;#5`#6`#7]{{
\settriparms[#1]
\diagram
\putptrianglep<\arrowtypea`\arrowtypeb`
\arrowtypec;\height>
(0,0)[#2`#3`#4;#5`#6`{#7}]
\enddiagram
}}                                            
 
\def\putqtrianglep<#1>(#2,#3)[#4`#5`#6;#7`#8`#9]{{%
\settriparms[#1]%
\xpos=#2 \ypos=#3
\advance\ypos by\height
\puthmorphism(\xpos,\ypos)[#4`#5`{#7}]{\height}{\arrowtypea}a%
\putmorphism(\xpos,\ypos)(1,-1)[``{#8}]{\height}{\arrowtypeb}l%
\advance\xpos by\height
\putvmorphism(\xpos,\ypos)[`#6`{#9}]{\height}{\arrowtypec}r%
}}
 
\def\putqtriangle{\@ifnextchar <{\putqtrianglep}{\putqtrianglep
   <\arrowtypea`\arrowtypeb`\arrowtypec;\height>}}
\def\qtriangle{\@ifnextchar <{\qtrianglep}{\qtrianglep
   <\arrowtypea`\arrowtypeb`\arrowtypec;\height>}}
\def\qtrianglep<#1>[#2`#3`#4;#5`#6`#7]{{
\settriparms[#1]
\width=\height                                
\diagram
\putqtrianglep<\arrowtypea`\arrowtypeb`
\arrowtypec;\height>
(0,0)[#2`#3`#4;#5`#6`{#7}]
\enddiagram
}}
 
\def\putdtrianglep<#1>(#2,#3)[#4`#5`#6;#7`#8`#9]{{%
\settriparms[#1]%
\xpos=#2 \ypos=#3
\puthmorphism(\xpos,\ypos)[#5`#6`{#9}]{\height}{\arrowtypec}b%
\advance\xpos by \height \advance\ypos by\height
\putmorphism(\xpos,\ypos)(-1,-1)[``{#7}]{\height}{\arrowtypea}l%
\putvmorphism(\xpos,\ypos)[#4``{#8}]{\height}{\arrowtypeb}r%
}}
 
\def\putdtriangle{\@ifnextchar <{\putdtrianglep}{\putdtrianglep
   <\arrowtypea`\arrowtypeb`\arrowtypec;\height>}}
\def\dtriangle{\@ifnextchar <{\dtrianglep}{\dtrianglep
   <\arrowtypea`\arrowtypeb`\arrowtypec;\height>}}
\def\dtrianglep<#1>[#2`#3`#4;#5`#6`#7]{{
\settriparms[#1]
\width=\height                                
\diagram
\putdtrianglep<\arrowtypea`\arrowtypeb`
\arrowtypec;\height>
(0,0)[#2`#3`#4;#5`#6`{#7}]
\enddiagram
}}
 
\def\putbtrianglep<#1>(#2,#3)[#4`#5`#6;#7`#8`#9]{{%
\settriparms[#1]%
\xpos=#2 \ypos=#3
\puthmorphism(\xpos,\ypos)[#5`#6`{#9}]{\height}{\arrowtypec}b%
\advance\ypos by\height
\putmorphism(\xpos,\ypos)(1,-1)[``{#8}]{\height}{\arrowtypeb}r%
\putvmorphism(\xpos,\ypos)[#4``{#7}]{\height}{\arrowtypea}l%
}}
 
\def\putbtriangle{\@ifnextchar <{\putbtrianglep}{\putbtrianglep
   <\arrowtypea`\arrowtypeb`\arrowtypec;\height>}}
\def\btriangle{\@ifnextchar <{\btrianglep}{\btrianglep
   <\arrowtypea`\arrowtypeb`\arrowtypec;\height>}}
\def\btrianglep<#1>[#2`#3`#4;#5`#6`#7]{{
\settriparms[#1]
\width=\height                               
\diagram
\putbtrianglep<\arrowtypea`\arrowtypeb`
\arrowtypec;\height>
(0,0)[#2`#3`#4;#5`#6`{#7}]
\enddiagram
}}
 
\def\putAtrianglep<#1>(#2,#3)[#4`#5`#6;#7`#8`#9]{{%
\settriparms[#1]%
\xpos=#2 \ypos=#3
{\multiply \height by2
\puthmorphism(\xpos,\ypos)[#5`#6`{#9}]{\height}{\arrowtypec}b}%
\advance\xpos by\height \advance\ypos by\height
\putmorphism(\xpos,\ypos)(-1,-1)[#4``{#7}]{\height}{\arrowtypea}l%
\putmorphism(\xpos,\ypos)(1,-1)[``{#8}]{\height}{\arrowtypeb}r%
}}
 
\def\putAtriangle{\@ifnextchar <{\putAtrianglep}{\putAtrianglep
   <\arrowtypea`\arrowtypeb`\arrowtypec;\height>}}
\def\Atriangle{\@ifnextchar <{\Atrianglep}{\Atrianglep
   <\arrowtypea`\arrowtypeb`\arrowtypec;\height>}}
\def\Atrianglep<#1>[#2`#3`#4;#5`#6`#7]{{
\settriparms[#1]
\width=\height                                     
\diagram
\putAtrianglep<\arrowtypea`\arrowtypeb`
\arrowtypec;\height>
(0,0)[#2`#3`#4;#5`#6`{#7}]
\enddiagram
}}
 
\def\putAtrianglepairp<#1>(#2)[#3;#4`#5`#6`#7`#8]{{%
\settripairparms[#1]%
\setpos(#2)%
\settokens`#3`%
\puthmorphism(\xpos,\ypos)[\tokenb`\tokenc`{#7}]{\height}{\arrowtyped}b%
\advance\xpos by\height
\puthmorphism(\xpos,\ypos)[\phantom{\tokenc}`\tokend`{#8}]%
{\height}{\arrowtypee}b%
\advance\ypos by\height
\putmorphism(\xpos,\ypos)(-1,-1)[\tokena``{#4}]{\height}{\arrowtypea}l%
\putvmorphism(\xpos,\ypos)[``{#5}]{\height}{\arrowtypeb}m%
\putmorphism(\xpos,\ypos)(1,-1)[``{#6}]{\height}{\arrowtypec}r%
}}
 
\def\putAtrianglepair{\@ifnextchar <{\putAtrianglepairp}{\putAtrianglepairp%
   <\arrowtypea`\arrowtypeb`\arrowtypec`\arrowtyped`\arrowtypee;\height>}}
\def\Atrianglepair{\@ifnextchar <{\Atrianglepairp}{\Atrianglepairp%
   <\arrowtypea`\arrowtypeb`\arrowtypec`\arrowtyped`\arrowtypee;\height>}}
 
\def\Atrianglepairp<#1>[#2;#3`#4`#5`#6`#7]{{
\settripairparms[#1]
\settokens`#2`
\width=\height                                
\diagram
\putAtrianglepairp                            
<\arrowtypea`\arrowtypeb`\arrowtypec`
\arrowtyped`\arrowtypee;\height>
(0,0)[{#2};#3`#4`#5`#6`{#7}]
\enddiagram
}}
 
\def\putVtrianglep<#1>(#2,#3)[#4`#5`#6;#7`#8`#9]{{%
\settriparms[#1]%
\xpos=#2 \ypos=#3
\advance\ypos by\height
{\multiply\height by2
\puthmorphism(\xpos,\ypos)[#4`#5`{#7}]{\height}{\arrowtypea}a}%
\putmorphism(\xpos,\ypos)(1,-1)[`#6`{#8}]{\height}{\arrowtypeb}l%
\advance\xpos by\height
\advance\xpos by\height
\putmorphism(\xpos,\ypos)(-1,-1)[``{#9}]{\height}{\arrowtypec}r%
}}
 
\def\putVtriangle{\@ifnextchar <{\putVtrianglep}{\putVtrianglep
   <\arrowtypea`\arrowtypeb`\arrowtypec;\height>}}
\def\Vtriangle{\@ifnextchar <{\Vtrianglep}{\Vtrianglep
   <\arrowtypea`\arrowtypeb`\arrowtypec;\height>}}
\def\Vtrianglep<#1>[#2`#3`#4;#5`#6`#7]{{
\settriparms[#1]
\width=\height                                 
\diagram
\putVtrianglep<\arrowtypea`\arrowtypeb`
\arrowtypec;\height>
(0,0)[#2`#3`#4;#5`#6`{#7}]
\enddiagram
}}
 
\def\putVtrianglepairp<#1>(#2)[#3;#4`#5`#6`#7`#8]{{
\settripairparms[#1]%
\setpos(#2)%
\settokens`#3`%
\advance\ypos by\height
\putmorphism(\xpos,\ypos)(1,-1)[`\tokend`{#6}]{\height}{\arrowtypec}l%
\puthmorphism(\xpos,\ypos)[\tokena`\tokenb`{#4}]{\height}{\arrowtypea}a%
\advance\xpos by\height
\puthmorphism(\xpos,\ypos)[\phantom{\tokenb}`\tokenc`{#5}]%
{\height}{\arrowtypeb}a%
\putvmorphism(\xpos,\ypos)[``{#7}]{\height}{\arrowtyped}m%
\advance\xpos by\height
\putmorphism(\xpos,\ypos)(-1,-1)[``{#8}]{\height}{\arrowtypee}r%
}}
 
\def\putVtrianglepair{\@ifnextchar <{\putVtrianglepairp}{\putVtrianglepairp%
    <\arrowtypea`\arrowtypeb`\arrowtypec`\arrowtyped`\arrowtypee;\height>}}
\def\Vtrianglepair{\@ifnextchar <{\Vtrianglepairp}{\Vtrianglepairp%
    <\arrowtypea`\arrowtypeb`\arrowtypec`\arrowtyped`\arrowtypee;\height>}}
\def\Vtrianglepairp<#1>[#2;#3`#4`#5`#6`#7]{{
\settripairparms[#1]
\settokens`#2`
\diagram
\putVtrianglepairp                             
<\arrowtypea`\arrowtypeb`\arrowtypec`
\arrowtyped`\arrowtypee;\height>
(0,0)[{#2};#3`#4`#5`#6`{#7}]
\enddiagram
}}

\def\putCtrianglep<#1>(#2,#3)[#4`#5`#6;#7`#8`#9]{{%
\settriparms[#1]%
\xpos=#2 \ypos=#3
\advance\ypos by\height
\putmorphism(\xpos,\ypos)(1,-1)[``{#9}]{\height}{\arrowtypec}l%
\advance\xpos by\height
\advance\ypos by\height
\putmorphism(\xpos,\ypos)(-1,-1)[#4`#5`{#7}]{\height}{\arrowtypea}l%
{\multiply\height by 2
\putvmorphism(\xpos,\ypos)[`#6`{#8}]{\height}{\arrowtypeb}r}%
}}
 
\def\putCtriangle{\@ifnextchar <{\putCtrianglep}{\putCtrianglep
    <\arrowtypea`\arrowtypeb`\arrowtypec;\height>}}
\def\Ctriangle{\@ifnextchar <{\Ctrianglep}{\Ctrianglep
    <\arrowtypea`\arrowtypeb`\arrowtypec;\height>}}
\def\Ctrianglep<#1>[#2`#3`#4;#5`#6`#7]{{
\settriparms[#1]
\width=\height                               
\diagram
\putCtrianglep<\arrowtypea`\arrowtypeb`
\arrowtypec;\height>
(0,0)[#2`#3`#4;#5`#6`{#7}]
\enddiagram
}}                                           
\def\putDtrianglep<#1>(#2,#3)[#4`#5`#6;#7`#8`#9]{{%
\settriparms[#1]%
\xpos=#2 \ypos=#3
\advance\xpos by\height \advance\ypos by\height
\putmorphism(\xpos,\ypos)(-1,-1)[``{#9}]{\height}{\arrowtypec}r%
\advance\xpos by-\height \advance\ypos by\height
\putmorphism(\xpos,\ypos)(1,-1)[`#5`{#8}]{\height}{\arrowtypeb}r%
{\multiply\height by 2
\putvmorphism(\xpos,\ypos)[#4`#6`{#7}]{\height}{\arrowtypea}l}%
}}
 
\def\putDtriangle{\@ifnextchar <{\putDtrianglep}{\putDtrianglep
    <\arrowtypea`\arrowtypeb`\arrowtypec;\height>}}
\def\Dtriangle{\@ifnextchar <{\Dtrianglep}{\Dtrianglep
   <\arrowtypea`\arrowtypeb`\arrowtypec;\height>}}
\def\Dtrianglep<#1>[#2`#3`#4;#5`#6`#7]{{
\settriparms[#1]
\width=\height                              
\diagram
\putDtrianglep<\arrowtypea`\arrowtypeb`
\arrowtypec;\height>
(0,0)[#2`#3`#4;#5`#6`{#7}]
\enddiagram
}}                                          
\def\setrecparms[#1`#2]{\width=#1 \height=#2}%
 
\def\recursep<#1`#2>[#3;#4`#5`#6`#7`#8]{{\m@th
\width=#1 \height=#2
\settokens`#3`
\settowidth{\tempdimen}{$\tokena$}
\ifdim\tempdimen=0pt
  \savebox{\tempboxa}{\hbox{$\tokenb$}}%
  \savebox{\tempboxb}{\hbox{$\tokend$}}%
  \savebox{\tempboxc}{\hbox{$#6$}}%
\else
  \savebox{\tempboxa}{\hbox{$\hbox{$\tokena$}\times\hbox{$\tokenb$}$}}%
  \savebox{\tempboxb}{\hbox{$\hbox{$\tokena$}\times\hbox{$\tokend$}$}}%
  \savebox{\tempboxc}{\hbox{$\hbox{$\tokena$}\times\hbox{$#6$}$}}%
\fi
\ypos=\height
\divide\ypos by 2
\xpos=\ypos
\advance\xpos by \width
\bfig
\putCtrianglep<-1`1`1;\ypos>(0,0)[`\tokenc`;#5`#6`{#7}]%
\puthmorphism(\ypos,0)[\tokend`\usebox{\tempboxb}`{#8}]{\width}{-1}b%
\puthmorphism(\ypos,\height)[\tokenb`\usebox{\tempboxa}`{#4}]{\width}{-1}a%
\advance\ypos by \width
\putvmorphism(\ypos,\height)[``\usebox{\tempboxc}]{\height}1r%
\efig
}}
 
\def\recurse{\@ifnextchar <{\recursep}{\recursep<\width`\height>}}
 
\def\puttwohmorphisms(#1,#2)[#3`#4;#5`#6]#7#8#9{{%
\puthmorphism(#1,#2)[#3`#4`]{#7}0a
\ypos=#2
\advance\ypos by 20
\puthmorphism(#1,\ypos)[\phantom{#3}`\phantom{#4}`#5]{#7}{#8}a
\advance\ypos by -40
\puthmorphism(#1,\ypos)[\phantom{#3}`\phantom{#4}`#6]{#7}{#9}b
}}
 
\def\puttwovmorphisms(#1,#2)[#3`#4;#5`#6]#7#8#9{{%
\putvmorphism(#1,#2)[#3`#4`]{#7}0a
\xpos=#1
\advance\xpos by -20
\putvmorphism(\xpos,#2)[\phantom{#3}`\phantom{#4}`#5]{#7}{#8}l
\advance\xpos by 40
\putvmorphism(\xpos,#2)[\phantom{#3}`\phantom{#4}`#6]{#7}{#9}r
}}
 
\def\puthcoequalizer(#1)[#2`#3`#4;#5`#6`#7]#8#9{{%
\setpos(#1)%
\puttwohmorphisms(\xpos,\ypos)[#2`#3;#5`#6]{#8}11%
\advance\xpos by #8
\puthmorphism(\xpos,\ypos)[\phantom{#3}`#4`#7]{#8}1{#9}
}}
 
\def\putvcoequalizer(#1)[#2`#3`#4;#5`#6`#7]#8#9{{%
\setpos(#1)%
\puttwovmorphisms(\xpos,\ypos)[#2`#3;#5`#6]{#8}11%
\advance\ypos by -#8
\putvmorphism(\xpos,\ypos)[\phantom{#3}`#4`#7]{#8}1{#9}
}}
 
\def\putthreehmorphisms(#1)[#2`#3;#4`#5`#6]#7(#8)#9{{%
\setpos(#1) \settypes(#8)
\if a#9 %
     \vertsize{\tempcounta}{#5}%
     \vertsize{\tempcountb}{#6}%
     \ifnum \tempcounta<\tempcountb \tempcounta=\tempcountb \fi
\else
     \vertsize{\tempcounta}{#4}%
     \vertsize{\tempcountb}{#5}%
     \ifnum \tempcounta<\tempcountb \tempcounta=\tempcountb \fi
\fi
\advance \tempcounta by 60
\puthmorphism(\xpos,\ypos)[#2`#3`#5]{#7}{\arrowtypeb}{#9}
\advance\ypos by \tempcounta
\puthmorphism(\xpos,\ypos)[\phantom{#2}`\phantom{#3}`#4]{#7}{\arrowtypea}{#9}
\advance\ypos by -\tempcounta \advance\ypos by -\tempcounta
\puthmorphism(\xpos,\ypos)[\phantom{#2}`\phantom{#3}`#6]{#7}{\arrowtypec}{#9}
}}
 
\def\setarrowtoks[#1`#2`#3`#4`#5`#6]{%
\def\toka{#1}
\def\tokb{#2}
\def\tokc{#3}
\def\tokd{#4}
\def\toke{#5}
\def\tokf{#6}
}
\def\hex{\@ifnextchar <{\hexp}{\hexp<1000`400>}}
\def\hexp<#1`#2>[#3`#4`#5`#6`#7`#8;#9]{%
\setarrowtoks[#9]
\yext=#2 \advance \yext by #2
\xext=#1 \advance\xext by \yext
\bfig
\putCtriangle<-1`0`1;#2>(0,0)[`#5`;\tokb``\tokd]
\xext=#1 \yext=#2 \advance \yext by #2
\putsquare<1`0`0`1;\xext`\yext>(#2,0)[#3`#4`#7`#8;\toka```\tokf]
\advance \xext by #2
\putDtriangle<0`1`-1;#2>(\xext,0)[`#6`;`\tokc`\toke]
\efig
}
\makeatother

\newtheorem{thm}{Theorem}
\newtheorem{prop}[thm]{Proposition}
\newtheorem{cor}[thm]{Corollary}
\newtheorem{lem}[thm]{Lemma}

\newenvironment{rem}[1]{\noindent {\em Remark.} #1}{}
\newenvironment{exa}[1]{\noindent {\em Example.} #1}{}

\newcommand{\ham}{\textup{Ham}}
\newcommand{\symp}{\textup{Symp}}
\newcommand{\fl}{\textup{Flux}}

\title{Bounded symplectic diffeomorphisms and
split flux groups}

\author{Carlos Campos-Apanco}
\address{CIMAT\\
      Jalisco S/N, Col. Valenciana \\
      Guanajuato, Gto., Mexico 36240}
\email{carlosca@cimat.mx}

\author{Andr\'es Pedroza}
\address{Facultad de Ciencias\\
         Universidad de Colima\\
     Bernal D\'{\i}az del Castillo No. 340\\
         Colima, Col., Mexico 28045}
\email{andres\_pedroza@ucol.mx}

\thanks{The authors were supported by CONACYT grant No. 50662 }

\begin{document}

\maketitle

\begin{abstract}
We prove the bounded isometry conjecture of F. Lalonde and L. Polterovich
for a special class of closed symplectic manifolds. 
As a byproduct,  it is shown that
the flux group of a product of these special symplectic manifold is isomorphic
to the direct sum of the flux group of each symplectic manifold.
\end{abstract}

\maketitle

\section{Introduction}

For a closed symplectic manifold $(M,\omega)$, the group $\ham(M,\omega)$ of Hamiltonian
diffeomorphisms  carries a norm called the Hofer norm.
The group $\ham(M,\omega)$  is a normal
subgroup of $\symp_0(M,\omega)$, the group of symplectic diffeomorphisms,
and the Hofer norm is invariant
under conjugation by $\symp_0(M,\omega)$.
For a fixed symplectic diffeomorphisms $\psi$,
the map $\mathcal{C}_\psi:\ham(M,\omega)
\to \ham(M,\omega)$ defined by $\mathcal{C}_\psi(h)=\psi\circ h\circ \psi^{-1}$ is an isometry
with respect to the Hofer norm.
In  \cite{lalonde-polterovich}
F. Lalonde and L. Polterovich   study the isometries
of the group of Hamiltonian diffeomorphisms with respect to the Hofer norm.
Based on this they call a symplectic diffeomorphism $\psi$
bounded if the Hofer norm of the commutator $[\psi, h]$ remains bounded as $h$ varies
in $\ham(M,\omega)$.  The set of bounded symplectic diffeomorphisms $\textup{BI}_0(M)$
of $(M,\omega)$ is a group that contains all Hamiltonian diffeomorphisms.

F. Lalonde and L. Polterovich conjectured that $\ham(M,\omega)
=\textup{BI}_0(M,\omega)$
for any closed symplectic manifold $(M,\omega)$. This problem is known as the bounded isometry conjecture.
In \cite{lalonde-polterovich} they proved  the conjecture 
 when the symplectic manifold is a surface of positive genus or is a product
of  these surfaces. In  \cite{lalonde-pe}  F. Lalonde and C. Pestieau  
proved the conjecture for the product of a closed surface of positive genus and a simply connected manifold.
Recently, Z. Han \cite{han2} proved the conjecture
for the  Kodaira--Thurston manifold.


In fact, in \cite{lalonde-polterovich} F. Lalonde and L. Polterovich proved a stronger result than
the bounded isometry conjecture.
They proved  that if  an equivalence class of $\symp_0(M,\omega)/\ham(M,\omega)$
has an unbounded symplectic diffeomorphism, then there is a strongly unbounded symplectic
diffeomorphism in the same class. 
This is equivalent to the fact that for
any nonzero element $v$ of $H^1(M)/\Gamma_M$ there is a strongly unbounded symplectic diffeomorphism with flux $v$. Here $\Gamma_M$ stands for the flux group of $(M,\omega)$. 
For the details, see Section \ref{s:Flux}. Here we  prove this stronger result.


We prove the bounded isometry conjecture for a closed symplectic manifold $(M,\omega)$ of dimension $2n$
satisfying the following two conditions:
\begin{itemize}
\item[(a)] There are open sets $U_1,\ldots, U_l\subset M$ such that each $U_k$ is symplectomorphic to
$\mathbb{T}^{2n}\setminus B(\epsilon_k)$ with the standard symplectic form. Here
$\mathbb{T}^{2n}$ is the $2n$--dimensional torus and
$B(\epsilon_k)$ is the embedded image of the standard closed ball in $\mathbb{R}^{2n}$ for a  sufficiently small $\epsilon_k\geq 0$.

\item[(b)]  Let $j_k: U_k\to M$ be the inclusion map and
$j_{k*}:  H_c^1(U_k) \to H^1(M)$ the induced map in cohomology. Then
$$
H^1(M)=\sum_{k=1}^l j_{k*}( H_c^1(U_k)).
$$
\end{itemize}
A symplectic manifold 
satifying the conditions above is said to satisfy {\em (H)}.
Unless otherwise stated, throughout this article 
cohomology $H^*(\cdot)$ stands for
de Rham cohomology and  $H_c^*(\cdot)$ stands for de Rham cohomology
with compact support.

\begin{thm}
\label{t:BIO}
Let $(M,\omega)$ be a closed symplectic manifold that satisfies (H).
Then $$
\textup{BI}_0(M,\omega)=\textup{Ham}(M,\omega).
$$
\end{thm}

The proof of Theorem \ref{t:BIO} is based on the fact that the bounded isometry conjecture
holds for the punctured torus $(\mathbb{T}^{2n}_*,\omega_0)$. This was shown
in \cite{lalonde-polterovich} for $n=1$, but in fact
their argument  works  for all $n$. For the sake of completeness we
prove that $\textup{BI}_0(\mathbb{T}^{2n}_*,\omega_0)=\ham^c(\mathbb{T}^{2n}_*,\omega_0)$
in Proposition \ref{p:strongunonT}. Here $\ham^c(M,\omega)$ stands for the group of
Hamiltonian diffeomorphisms of $(M,\omega)$ with compact support.

The first example  of a closed symplectic manifold satisfying {\em (H)} is 
a closed surface $(\Sigma_g,\omega)$   with $g$ embedded punctured tori. 
Another example is the blow-up of the torus $(\mathbb{T}^{2n},\omega_0)$ at one point, 
or more
 generally the blow-up of $(\mathbb{T}^{2n},\omega_0)$ along a 
simply connected symplectic submanifold.
In Section 2, we give some more examples of symplectic manifolds that satisfy {\em (H)}.



The bounded isometry conjecture holds
for a wider class of symplectic manifolds that just those that satisfy {\em (H)}.

\begin{cor}
\label{c:prodbounded}
Let $(M,\omega)$  be a closed symplectic manifold that satisfies (H), 
and $(N,\eta)$ a closed symplectic manifold such that $H^1(N)$ is trivial or
satisfies (H). Then
$$
\textup{BI}_0(M\times N,\omega\oplus\eta)=\ham(M\times N,\omega\oplus\eta).
$$
\end{cor}

As a consequence of our argument in the proof of the bounded isometry conjecture for this
particular class of manifolds, we obtain
an interesting result about the flux group. We  show  that the flux group
of a product of two closed symplectic manifolds is isomorphic to
the direct sum  of the flux group of each manifold. That is, if
$(M,\omega)$ and $(N,\eta)$ are symplectic manifolds as in Corollary \ref{c:prodbounded}
with flux groups $\Gamma_M$ and $\Gamma_N$, then 
 $\Gamma_{M\times N}\simeq\Gamma_M\oplus\Gamma_N$,
where $\Gamma_{M\times N}$ is the flux group of  $(M\times N, \omega\oplus\eta)$.
When this relation holds, we  say that the flux group of
$(M\times N, \omega\oplus\eta)$ splits.

For instance this is well-known  when we consider copies
of $(\mathbb{T}^2,\omega_0)$. In this case, $\Gamma_{\mathbb{T}^2}$
is equal to $H^1(\mathbb{T}^2,\mathbb{Z})$ and direct calculation
 shows that $\Gamma_{\mathbb{T}^{2n}}=\Gamma_{\mathbb{T}^2}
\oplus\cdots \oplus\Gamma_{\mathbb{T}^2}$ (see 
\cite[Ch.\ 10]{mc-sal},  and (\cite[Ch.\ 14]{pol}).
Recall  that  in \cite{kedra} J. K\c{e}dra
gave conditions under which the flux group vanishes and these conditions
are compatible with products. For instance, if $(M,\omega)$
is aspherical with nonzero Euler characteristic, then by Theorem B of \cite{kedra},
 $\Gamma_M\oplus\Gamma_M=\Gamma_{M\times M}=0.$

When $(M,\omega)$ is a closed surface of genus greater  than one,
 the flux group is trivial; when $g=1$, the flux group equals $\mathbb{Z}^2$ .
In \cite[Remark 4.3.E]{lalonde-polterovich}, F. Lalonde and L. Polterovich
showed that the flux group splits  when the manifold is a product of
closed surfaces of positive genus. They achieved this in their study of
bounded symplectic diffeomorphisms.  Here we follow closely their ideas.

\begin{thm}
\label{t:prodflux}
Let $(M,\omega)$ and $(N,\eta)$  be closed symplectic manifolds as in
Corollary \ref{e:exacsecfluxfundamental}.
Then the flux group of $(M\times N,\omega\oplus\eta)$ splits:
$\Gamma_{M\times N}\simeq\Gamma_{M}\oplus\Gamma_{N}.$
\end{thm}

In Section \ref{s:ejemplos} we give an application of Theorem \ref{t:prodflux}
to the fundamental group of $\ham(S^2\times\Sigma_g)$.
Finally we point out an equivalent statement to that of Theorem \ref{t:prodflux}.

\begin{thm}
\label{t:prodHam}
Let $(M,\omega)$ and $(N,\eta)$  be closed symplectic manifolds as in
Corollary \ref{e:exacsecfluxfundamental}. 
 If $\psi\in \symp_0(M,\omega)$ and $\phi\in \symp_0(N,\eta)$ are  such that
$\psi \times \phi$ is a Hamiltonian diffeomorphism of $(M\times N,\omega\oplus\eta)$, then
$\psi$ and $\phi$ are  Hamiltonian diffeomorphisms.
\end{thm}

Finally we make the remark that all  the results remain true in the noncompact case,
as long as one considers diffeomorphisms with compact support.
\medskip

The authors thank Pro. K. Ono for helpful comments on the first draft of this note, and to 
Prof. D  Ruberman and Prof. L. Tu for their valuable comments on improving the exposition 
of this note. The second author  wishes to thank ICTP, Trieste for its hospitality during 
part of the work on this paper.



\section{Examples}
\label{s:ejemplos}

\begin{exa}
 Consider the torus $(\mathbb{T}^{2n},\omega_0)$ with its standard symplectic form. Let 
 $(\tilde{\mathbb{T}}^{2n},\tilde{\omega}_0)$ be its blow-up 
at one point. See \cite{mcduff-Examples}. There is a small $\epsilon>0$ such that
the inclusion  $\mathbb{T}^{2n} \setminus B(\epsilon) 
\to \tilde{\mathbb{T}}^{2n}$ is a symplectic embedding. Moreover, the 
induced map
$H^1_c(\mathbb{T}^{2n} \setminus B(\epsilon) ) \to 
H^1(\tilde{\mathbb{T}}^{2n})$ is an isomorphism. Hence, $(\tilde{\mathbb{T}}^{2n},\tilde{\omega}_0)$
satisfies {\em (H)} and  $\textup{BI}_0(\tilde{\mathbb{T}}^{2n},\tilde{\omega}_0)
=\textup{Ham}(\tilde{\mathbb{T}}^{2n},\tilde{\omega}_0)$.

More generally, let $N$ be a  simply connected symplectic submanifold of $(\mathbb{T}^{2n},\omega_0)$.
Denote by $(\tilde{\mathbb{T}}^{2n}_N,\tilde{\omega}_0)$ the blow up of $(\mathbb{T}^{2n},\omega_0)$
along $N$. Since $N$ is simply connected, the blow up map $\tilde{\mathbb{T}}^{2n}_N\to 
{\mathbb{T}}^{2n}$ 
induces an isomorphism $H^1(  {\mathbb{T}}^{2n})\to H^1(\tilde{\mathbb{T}}^{2n}_N)$.
Therefore, $\textup{BI}_0(\tilde{\mathbb{T}}^{2n}_N,\tilde{\omega}_0)
=\textup{Ham}(\tilde{\mathbb{T}}^{2n}_N,\tilde{\omega}_0)$ by Theomre \ref{t:BIO}.
\end{exa}

\medskip
Thus, in every dimension we have new examples of symplectic manifolds that
satisfy the bounded isometry conjecture. The next example explores some consequences
of Theorem \ref{t:prodflux}.

\begin{exa}
Consider the symplectic embedding $(\mathbb{T}^{2n},\omega_0)\to (\mathbb{T}^{2(m+n)}, \omega) $ 
in the last $2n$ coordinates. The symplectic  form on the torus is the canonical symplectic form.
Thus $\mathbb{T}^{2(m+n)}\setminus\mathbb{T}^{2n}=\mathbb{T}^{2m}_*\times\mathbb{T}^{2n}$.
It follows by  Corollary \ref{c:prodbounded} that 
$\textup{BI}_0( \mathbb{T}^{2(m+n)}\setminus\mathbb{T}^{2n},{\omega})
=\textup{Ham}^c( \mathbb{T}^{2(m+n)}\setminus\mathbb{T}^{2n} ,\omega)$.

It is also possible to show directly that 
$(\mathbb{T}^{2(m+n)}\setminus\mathbb{T}^{2n},{\omega})$ satisfies the bounded isometry conjecture.
Our  arguments in the proof of Proposition \ref{p:strongunonT}
apply to  this case with no major changes.
In fact, condition {\em (H)} can be weakened by allowing  the set $U$ to
be symplectomorphic to $\mathbb{T}^{2(m+n)}\setminus \mathbb{T}^{2n}$,
and not only symplectomorphic to a punctured torus. 
\end{exa}


\medskip
\begin{exa}
Let  $(S^2,\omega)$ be the $2$--sphere and $(\Sigma_g,\eta)$ a Riemman surface
of genus $g\geq 1$, each with a symplectic form of total area $1$.
Recall that
$\pi_1 (\ham (S^2,\omega))\simeq \mathbb{Z}_2$ and  $\ham (\Sigma_g,\eta)$
is simply connected for $g\geq 1$. By \cite{Remi} and  \cite{pedroza}
we can  say that $\pi_1(\ham(S^2\times \Sigma_g,\omega\oplus\eta))$
has an element of order two. We can say more by using Theorem \ref{t:prodflux}.

Since the flux group,
$\Gamma_{S^2}$ is trivial, by Theorem \ref{t:prodflux},  it follows 
 $\Gamma_{S^2\times\Sigma_g}= \Gamma_{\Sigma_g}$. Thus,
$\Gamma_{S^2\times\Sigma_g}$ equals $\mathbb{Z}\oplus\mathbb{Z}$
for $g=1$ and is trivial for $g>1$. It follows from
the exact sequence (\ref{e:exacsecfluxfundamental}) below
that for $g\geq 1$ the inclusion map induces an isomorphism
$$
\pi_1(\ham(S^2\times \Sigma_g,\omega\oplus\eta)) \to
\pi_1(\symp_0(S^2\times \Sigma_g,\omega\oplus\eta)).
$$
For $g=1$ we get  the exact sequence
\begin{eqnarray}
\label{e:exacseqS}
0\to \pi_1(\ham(S^2\times \mathbb{T}^2,\omega\oplus\eta)) \to
\pi_1(\symp_0(S^2\times \mathbb{T}^2,\omega\oplus\eta)) \to
\mathbb{Z}\oplus\mathbb{Z} \to 0.
\end{eqnarray} 
Let $\mathcal{D}^g$ denote the group of volume-preserving diffeomorphisms of $S^2\times \Sigma_g$  that also
preserve the fibers of $S^2\times \Sigma_g\to S^2$. According to D. McDuff  \cite[Prop. 1.6]{mcduff:almost},
for  $g=1$, the
 map $\pi_1(\symp_0(S^2\times \mathbb{T}^2,\omega\oplus\eta))\to \pi_1(\mathcal{D}^1)$ is
an isomorphism and $g>1$, the map $\pi_1(\symp_0(S^2\times \Sigma_g,\omega\oplus\eta))
\to \pi_1(\mathcal{D}^g)$ is surjective. Moreover by \cite[ Cor. 5.4]{mcduff:almost},
 $\pi_1(\mathcal{D}^g)\otimes \mathbb{Q}$ has dimension three when $g=1$ and dimension
one when $g>1$. Hence, from the exact sequence (\ref{e:exacseqS}) we get the exact sequence
\begin{eqnarray*}
\pi_1(\ham(S^2\times \mathbb{T}^2,\omega\oplus\eta))\otimes\mathbb{Q} \to
\pi_1(\symp_0(S^2\times \mathbb{T}^2,\omega\oplus\eta))\otimes\mathbb{Q} \to
(\mathbb{Z}\oplus\mathbb{Z})\otimes\mathbb{Q} \to 0
\end{eqnarray*} 
We conclude that the dimension of $\pi_1(\ham(S^2\times \Sigma_g,\omega\oplus\eta))\otimes\mathbb{Q}$
is at least one for $g\geq 1$.
\end{exa}

\section{The flux morphism}
\label{s:Flux}

First a word of warning: if $G$ is a group of diffeomorphisms, we will use $\psi$ to
denote an element in $\pi_1(G)$ and also
to denote a diffeomorphism. It will be clear
from the context what it represents.

Let $(M,\omega)$ be a closed symplectic manifold and $\psi= \{\psi_t\}_{0\leq t\leq 1 }$
a loop that represents an element of $\pi_1(\symp_0(M,\omega))$. The isotopy $\{\psi_t\}$ induces a
time-dependent vector field $X_t$ given by  the equation
$$
\frac{d}{dt}\psi_t=X_t\circ \psi_t.
$$
Then the flux morphism $\textup{Flux}_M: \pi_1(\symp_0(M,\omega))\to H_{dR}^1(M)$
is defined by
$$
\textup{Flux}_M(\psi)=\int_0^1 [\iota(X_t)\omega]dt.
$$
This map is well defined, that is, it depends only on the homotopy class in $\symp_0(M,\omega)$
based at the identity,  and  is a group morphism.  The image of $\textup{Flux}_M$ is
denoted by  $\Gamma_M$ and is called the {\em flux group} of $(M,\omega).$
The rank of $\Gamma_M$ is bounded by the first Betti number $b_1(M)$ and is a discrete subgroup of
$H^1(M)$ (see \cite{lmp} and \cite{ono}).
Moreover, the flux
morphism fits into the exact sequence of abelian groups
\begin{eqnarray}
\label{e:exacsecfluxfundamental}
0\to \pi_1(\ham(M,\omega)) \to \pi_1(\symp_0(M,\omega))
\to \Gamma_M \to 0
 \end{eqnarray}
where the first map is induced by inclusion and the next one is the flux morphism.

The flux morphism can also be defined on $\symp_0(M,\omega)$, rather than on its fundamental group.
In this case for a given symplectic diffeomorphism $\psi$ one considers a symplectic isotopy that joints
$1_M$ with $\psi$; this will induced a time-dependent vector field $X_t$ as before. As 
in the previous case we have the map  $\textup{Flux}_M:\symp_0(M,\omega) \to
H^1(M)/\Gamma_M$. There is also an exact sequence for this morphism,
\begin{eqnarray}
\label{e:exacsecflux}
0\to \ham(M,\omega) \to \symp_0(M,\omega)
\to H^1(M)/\Gamma_M \to 0,
 \end{eqnarray}
where the first map is inclusion and the last one is the flux morphism  just defined.
Note that  if $\psi$ and $\phi$ are symplectic diffeomorphisms
with the same flux, then by the exact sequence (\ref{e:exacsecflux}) there is a Hamiltonian 
diffeomorphism $\theta$ such that $\psi=\phi\circ\theta.$  This observation will be used later.

Finally,  the flux morphism can also be defined for noncompact symplectic manifolds.
In this case one considers symplectic diffeomorphisms with compact support, and the flux
morphism takes the form $\textup{Flux}_M: \pi_1(\symp_0^c(M,\omega))\to H^1_c(M)$,
and similarly for the flux defined on the group $\symp_0^c(M,\omega)$.
For more details of the flux morphism see  the books of D. McDuff and D. Salamon \cite{mc-sal} and of
L. Polterovich \cite{pol}.

\bigskip
Consider two closed symplectic manifolds $(M,\omega)$ and $(N,\eta)$. Then
$(M\times N,\omega\oplus\eta)$,  where $\omega\oplus\eta$
stands for $\pi_M^*(\omega)+\pi_N^*(\eta)$, is also a symplectic manifold. The map
$$
\Psi : \symp_0(M,\omega)\times \symp_0(N,\eta)\to \symp_0(M\times N,\omega\oplus\eta)
$$
given by $\Psi(\psi,\phi)=\psi\times\phi$ is a well-defined group homomorphism.
It also follows  that  $\Gamma_{M}\oplus\Gamma_{N}$ is a subgroup of
$\Gamma_{M\times N}$, 
 so  the induced map $i_0: H^1(M\times N)/\Gamma_{M}\oplus\Gamma_{N} \to H^1(M\times N)/\Gamma_{M\times N}$
is surjective.  Tp prove that $\Gamma_{M\times N}\simeq\Gamma_{M}\oplus\Gamma_{N}$, it suffices
to show that the map $i_0$ is injective. We can rephrase this in terms of Hamiltonian diffeomorphisms via the
exact sequence of the flux morphism.

\begin{lem}
\label{l:prodHam}
Theorems \ref{t:prodflux} and \ref{t:prodHam} are equivalent.
\end{lem}
\begin{proof}
This follows by analyzing the exact sequence (\ref{e:exacsecflux}) of the flux morphism. We use the
exact sequence (\ref{e:exacsecflux}) for the manifolds $(M,\omega), (N,\eta)$ and
$(M\times N,\omega\oplus\eta)$ as in the next diagram where the rows are exact.
\begin{center}
\begin{tabular}{ccccccc}
 $\cdots$&$\to$& $\symp_0(M,\omega)\oplus \symp_0(N,\eta)
$&$\longrightarrow$& $H^1(M\times N)/\Gamma_{M}\oplus\Gamma_{N}$&$\longrightarrow$& $0$ \\
 & &$\downarrow$ &  & $\downarrow$ & & \\
 $\cdots$&$\to$&$ \symp_0(M\times N,\omega\oplus\eta)
$&$\longrightarrow$&$ H^1(M\times N)/\Gamma_{M\times N}$&$ \longrightarrow$&$ 0$ \\
\end{tabular}
\end{center}
 Here the horizontal maps are $\textup{Flux}_M\oplus \textup{Flux}_N$ and $\textup{Flux}_{M\times N}.$
And the vertical maps are $\Psi$ and $i_0$. So defined, the diagram commutes and the lemma follows.
\end{proof}

This lemma is the link between the theory of bounded symplectic diffeomorphisms and
 our question about the splitting of the flux group.

\section{Bounded symplectic diffeomorphisms}

Recall that the group $\ham(M,\omega)$ of Hamiltonian diffeomorphisms
is a normal subgroup of the group $\symp(M,\omega)$ of symplectic diffeomorphisms.
A symplectic diffeomorphism $\psi$ is called {\em bounded} if the
set
$$
\{ \| [\psi, f] \| : f\in \ham(M,\omega)   \}
$$
is bounded. Here $\|\cdot \|$ stands for the Hofer norm on
$\ham(M,\omega)$. A symplectic diffeomorphism is called {\em unbounded}
if is not bounded. The set of bounded symplectic diffeomorphisms
forms a subgroup of $\symp(M,\omega)$ and is denoted by $\textup{BI}(M,\omega)$.

Since for any $\psi\in \ham(M,\omega)$ and $f\in\ham(M,\omega)$, we have 
$\|[\psi,f]\|\leq 2 \| \psi\|$, every Hamiltonian diffeomorphism
is a bounded diffeomorphism. Thus $\ham(M,\omega)$ is a subgroup
of $\textup{BI}(M,\omega)$. Define
$$
\textup{BI}_0(M,\omega)=\textup{BI}(M,\omega) \cap   \symp_0(M,\omega).
$$

In this section we  generalize the work \cite{lalonde-polterovich} of F. Lalonde and
L. Polterovich, in which  a
fundamental observation was that $\textup{BI}_0(\mathbb{T}^{2}\setminus \{pt\},\omega_0)=
\ham^c(\mathbb{T}^{2}\setminus \{pt\},\omega_0)$.
A symplectic diffeomorphism $\psi$ of $(M,\omega)$ is called
{\em strongly unbounded} if for every $c>0$ there is an $f\in\ham(M,\omega)$
such that the lift of $[\psi,f]$ to $\tilde M$ disjoins a ball of
capacity equal to $c$ from itself. Here $\tilde M$ stands for the universal cover of $M$.

Recall that
the universal cover $(\tilde M,\tilde\omega)$ of $(M,\omega)$ is also a 
symplectic manifold; moreover the projection map $\pi\colon\tilde M\to M$
satisfies $\pi^*(\omega)=\tilde \omega.$  Let $\psi$ be a Hamiltonian diffeomorphism
of $(M,\omega)$ and  $H_t\colon M\to \mathbb{R}$  a Hamiltonian function, whose time-one flow is $\psi$. Then
$H_t\circ\pi$ generates a Hamiltonian flow on $(\tilde M,\tilde\omega)$,
with time-one map $\tilde \psi.$ So defined $\tilde \psi$ is a Hamiltonian diffeomorphism
that lifts $\psi$. According to Z. Han \cite[Lemma 2.1]{han}, 
 every Hamiltonian diffeomorphism
has a unique lift to $(\tilde M,\tilde \omega)$.

The concepts of strongly unbounded symplectic diffeomorphisms and lifts
of Hamiltonian diffeomorphisms are fundamental in the proof of the bounded
isometry conjecture. 
The reason is that using them one
can get large lower bounds for the Hofer norm. By the energy-capacity inequality, 
if $\psi$ is a Hamiltonian diffeomorphism  such that $\psi(A)\cap A=\emptyset$
for $A\subset M$, then
$$
\frac{1}{2}  c_G(A)\leq e(A)\leq \|\psi\|.
$$
Here $e(A)$ is the displacement energy of $A$ and $c_G(A)$  Gromov's capacity
of $A$ ( see \cite{lalonde-mcduff-energy}).
However, this inequality is not enough for closed symplectic manifolds, since
 the capacity $c_G(\cdot)$ is bounded from above. Hence we need to pass to the
universal cover of the symplectic manifold, since on this open symplectic manifold there
are subsets with arbitrary large capacity.

\begin{prop}[Prop. 1.5A in \cite{lalonde-polterovich}]
\label{p:Hoferboundbelow}
If $\psi$ is a Hamiltonian diffeomorphisms of $(M, \omega)$ whose unique lift 
$\tilde\psi:\tilde M\to \tilde M$ disjoins
a ball of capacity $c$ from itself, then $\|\psi\|\geq c/2.$
\end{prop}

We will show that the bounded isometry conjecture holds for $\mathbb{T}^{2n}\setminus B(\epsilon_0)$. 
In dimension
two  this was proved   by F. Lalonde and L. Polterovich in \cite{lalonde-polterovich}. Our proof
is just an extension of their arguments. 

We review a couple of facts of \cite{lalonde-polterovich} that we need in the proof of the next
proposition. 
In order to have a clear exposition of the arguments,
instead of considering diffeomorphisms of  $\mathbb{T}^{2n}\setminus B(\epsilon_0)$,
we will consider periodic diffeomorphisms of $\mathbb{R}^{2n}$ minus a small ball centered
at every point of $\mathbb{Z}^{2n}$.
So let $a$ be a small positive number greater than $\epsilon_0$; and for each
 $(k_1,l_1,\ldots,k_n,l_n)\in \mathbb{Z}^{2n}$, consider the small box
$\{ (x_1,y_1,\ldots, x_n,y_n): |x_j- k_j|\leq 3a  \mbox{ and  } |y_j- l_j|\leq 3a \}$.
Denote by $W$ the union of all such boxes as the point $(k_1,l_1,\ldots,k_n,l_n)$ varies
in $\mathbb{Z}^{2n}$.
Let $p:\mathbb{R}\to \mathbb{R}$
 be any
smooth 1-periodic function such that
$$
p =
\left\{
	\begin{array}{ll}
		0  & \mbox{on } [0,4a-2\epsilon],  [4a+2\epsilon, 5a-2\epsilon] \mbox{ and }  [5a+2\epsilon,1] \\
		-1 & \mbox{on }[5a-\epsilon, 5a+\epsilon] \\
		1 & \mbox{on } [4a-\epsilon, 4a+\epsilon] \\
		\mbox{monotone} & \mbox{ on the remaining subintervals of } [0,1].
	\end{array}
\right.
$$
 Here $\epsilon$ is a positive number so
small  that the definition of $p$ make sense.
Finally we also require that
$$
\int_0^1 p(s)ds=0.
$$


\begin{prop}
\label{p:strongunonT}
For any nonzero $v$ in $H_c^1(\mathbb{T}^{2n}\setminus B(\epsilon_0))/
\Gamma_{\mathbb{T}^{2n}\setminus B(\epsilon_0)}$, there is  a
symplectic diffeomorphism $\theta$ with compact support in $\mathbb{T}^{2n}\setminus B(\epsilon_0)$ that is
strongly unbounded and with $\textup{Flux}(\theta)=v$. 
In particular $\textup{BI}_0(\mathbb{T}^{2n}\setminus B(\epsilon_0),\omega)=
\ham^c(\mathbb{T}^{2n}\setminus B(\epsilon_0),\omega)$.
\end{prop}
\begin{proof}
Since $H^1_c(\mathbb{T}^{2n}\setminus B(\epsilon_0))=\mathbb{R}^{2n}$ we can find
 generators $e_1,f_1,\dots, e_n,f_n$ of $H^1_c(\mathbb{T}^{2n}\setminus B(\epsilon_0))$ that
are dual to the canonical cycles of the torus. Also let $(a_1,b_1,\ldots,a_n,b_n)$ be a $2n$-tuple of
non negative real numbers not all of which are zero.
We will define $\psi_i,\phi_i\in
\symp^c_0(\mathbb{T}^{2n}\setminus B(\epsilon_0))$ with flux $a_ie_i$ and $b_if_i$ respectively.

Let $a$ and $W$ as above. Consider  
a smooth  1-periodic function $h_{a_j}:\mathbb{R}\to \mathbb{R}$  such that it is
equal to zero on $[0,1/3]$ and $[2/3,1]$, is positive otherwise,
 and satisfies
\begin{eqnarray}
\label{e:integral}
a+\int_0^1 h_{a_j}(s)ds=a_j.
\end{eqnarray}
Similarly for each $j$ we have a function $h_{b_j}$ satistying the same properties with $b_j$ instead of
$a_j$.
Then  define the symplectic diffeomorphisms
$\psi_j$ and $\phi_j$ of $(\mathbb{R}^{2n},\omega_0)$ as
$$
\psi_j(x_1,y_1,\ldots,x_n,y_n)=  (x_1,\ldots,\ \ x_j,\ \ y_j+a+h_{a_j}(x_j), \ \ \ldots,y_n)
$$
and
$$
\phi_j(x_1,y_1,\ldots,x_n,y_n)=  (x_1,\ldots,\ \ x_j+a+h_{b_j}(y_j),\ \  y_j,\ \ \ldots,y_n)
$$
outside $W$ and  fix the points close to each $B(\epsilon_0)$.
The maps $\psi_j$ correspond to the time-one map of the symplectic flow
$$
(x_1,y_1,\ldots,x_n,y_n)\mapsto   (x_1,\ldots, \ \ x_j,\ \ y_j+t(a+h_{a_j}(x_j)),\ \ \ldots,y_n)
$$
and similarly for $\phi_j$.

Basically the  maps $\psi_j$ and $\phi_j$  are translations along the $y_j$-axes and
$x_j$-axes  of $\mathbb{R}^{2n}$ respectively.
From Equation (\ref{e:integral}) it follows that $\textup{Flux}(\psi_j)=a_j$
and $\textup{Flux}(\phi_j)=b_j$.
Then the flux of
$\psi_1\circ \phi_1\circ\cdots\circ \psi_n\circ \phi_n$ is equal to
$a_1e_1+ b_1f_1+\cdots+ a_ne_n +b_nf_n$.  Recall 
$a_j$ and $b_j$ are assumed to  be non negative. If $a_j$ is zero, we define
$\psi_j$ to be  the identity diffeomorphism. If $a_j$ is negative,
we proceed as above  with $-a_j$ instead of $a_j$ and then 
$\textup{Flux}(\psi_j^{-1})=a_j$

We claim that the symplectic
diffeomorphism  $\theta=\psi_1\circ \phi_1\circ\cdots\circ \psi_n\circ \phi_n$ is strongly unbounded.
To see this, consider the symplectic isotopy
$$
f_t(x_1,y_1,\ldots,x_n,y_n)= (x_1,\ \ y_1+tp(x_1), \ldots,x_n,y_n)
$$
where $p:\mathbb{R}\to \mathbb{R}$ is the 1-periodic function defined above.
Since $p$ vanishes on $[0,4a-2\epsilon]$ and $[5a+2\epsilon,1]$, each $f_t$ leaves 
$W$ fixed pointwise.
The zero mean condition on $p$, implies that $\{f_t\}$ is a Hamiltonian isotopy.
Since $f_t$ commutes with $\phi_1,\ \ \psi_2,\ldots, \psi_n$ and $\phi_n$ but not with
 $\psi_1$, we have 
 $[\theta, f_t]=[\psi_1, f_t]$. Note that  $[\theta, f_t]=[\psi_1, f_t]$
is the identity on the last $2n-2$ coordinates of $\mathbb{R}^{2n}$, and in then
$(x_1,y_1)$-plane it corresponds to the symplectic diffeomorphism $g_t$ 
constructed in \cite{lalonde-polterovich},
that is,
$$
[\theta, f_t]=[\psi_1, f_t]=g_t\times 1_{\mathbb{R}^{2n-2}},
$$

Recall from \cite{lalonde-polterovich} that $g_t$ disjoins a rectangle $B_t$
whose area is a function of $t$. Therefore 
$[\theta, f_t](B_t\times \mathbb{R}^{2n-2})\cap (B_t\times \mathbb{R}^{2n-2})=\emptyset$. 
In $\mathbb{R}^2$ the rectangle
$B_t$ is symplectomorphic to a disk of the same area. Since the area of $B_t$ goes to infinity
as $t$ goes to infinity, by the energy-capacity inequality 
the Hofer norm of $[\theta, f_t]=[\psi_1, f_t]$ goes to infinity as $t$ goes to infinity.
Hence $\theta\in \symp^c(\mathbb{T}^{2n}\setminus B(\epsilon_0),\omega)$ 
is strongly unbounded.
\end{proof}

\begin{rem}
It is important to note from the proof of Proposition \ref{p:strongunonT}
that the  $(x_1,y_1)$-plane of $\mathbb{R}^{2n}$  and the Hamiltonian
isotopy $\{f_t\}$ are not related at all to  $v\in H_c^1(\mathbb{T}^{2n}\setminus B(\epsilon_0))/
\Gamma_{\mathbb{T}^{2n}\setminus B(\epsilon_0)}$.
This observation will be useful when we generalize this result to closed symplectic manifolds
that satisfy hypothesis {\em (H)}.
\end{rem}

Before we extend the previous result to  symplectic manifolds that satisfy  {\em (H)}
we need the following lemma. It will be used in order  to show that the strongly unbounded
 diffeomorphism $\theta$ defined in  proof of Proposition \ref{p:strongunonT}, would remain
 strongly unbounded  on $(M,\omega)$ and not only on the open manifold 
$\mathbb{T}^{2n}\setminus B(\epsilon_0)$.

\begin{lem}
\label{l:embeddingstrip}
Let $(M,\omega)$ be a closed manifold with nontrivial $H^1(M)$. Then
there is a symplectic embedding of $((0,\epsilon)\times \mathbb{R}, dx\wedge dy)$
into $(\tilde M,\tilde \omega)$, where $\epsilon>0$ is small.
\end{lem}
\begin{proof}
Since $H^1(M)$ is nontrivial, there is an embedding $i:\mathbb{R}\to \tilde M$.
Moreover since $\mathbb{R}$ is contractible, the normal bundle $\nu$ is
isomorphic to $\mathbb{R}^{2n-1}\times\mathbb{R}$. Put the canonical symplectic
for on $\nu$. Then there is symplectic diffeomorphims between a neighborhood
of the zero section of $\nu$ and a neighborhood of $i(\mathbb{R})$ in $\tilde M$.
It follows that $\ (0,\epsilon)\times \mathbb{R}$ embeds symplectically into $\tilde M$.
\end{proof}

Since a symplectic diffeomorphism with compact support in $U$ can be thought of as a
symplectic diffeomorphism on $M$,
there is a natural map $\tau:\symp^c_0(U,\omega)\to \symp_0(M,\omega)$.
This gives rise to the  commutative diagram
$\fl_M\circ \tau=j_*\circ \fl_U$, where $j_*: H^1_c(U)\to H^1(M).$
Hence  $j_*(\Gamma_U)$ is a subgroup of $\Gamma_M$.


Then from the commutative diagram
\begin{center}
\begin{picture}(2500,700)(0,0)
\put(60,500){$\symp^c_0(U,\omega) \hskip 3.3cm
  \symp_0(M,\omega)$}
\put(100,10){$H_{c}^1(U)/\Gamma_U \hskip 3.6cm
  H^1(M)/\Gamma_M$}
\put(700,40){\vector(1,0){800}}
\put(790,520){\vector(1,0){540}}
\put(350,430){\vector(0,-1){260}}
\put(1800,430){\vector(0,-1){260}}
\put(50,250){\begin{small}$\fl_U$ \end{small}}
\put(1900,250){\begin{small}$\fl_M $\end{small}}
\put(1000,600){\begin{small}$\tau$ \end{small}}
\put(1000,120){\begin{small}$\tilde j_*$ \end{small}}
\end{picture}
\end{center}
and Proposition \ref{p:strongunonT} we have the following result.

\begin{prop}
\label{p:stromgunbM}
Let $(M,\omega)$ be a closed symplectic manifold that satisfies hypothesis (H).
Then for any nonzero $v$  in $H^1(M)/\Gamma_M$, there is  a
symplectic diffeomorphism $\psi$ that is strongly unbounded and
$\textup{Flux}(\psi)=v$. In
particular Thm. \ref{t:BIO} holds: $\textup{BI}_0(M,\omega)=\ham(M,\omega)$.
\end{prop}
\begin{proof}
Let $v$ in $H^1(M)/\Gamma_M$. Since $(M,\omega)$ satisfies {\em  (H)}, 
we have that $v=v_1+\cdots +v_l$ where $v_r\in j_{r,*}(H^1_c(U_r))$. For
simplicity assume $l=1.$ Thus there is
an open set $U\subset M$  that is symplectomorphic to $\mathbb{T}^{2n}\setminus B(\epsilon_0)$  and
$v_0$ in $H^1_c(U)/\Gamma_U$ nonzero such that $\tilde j_*(v_0)=v$
under the inclusion map $j:U\to M$.
By Proposition \ref{p:strongunonT} there is $\psi\in \symp_0^c(U,\omega)$ 
strongly unbounded and
flux equal to $v_0$.

Consider $\psi$ as a symplectic diffeomorphism in $\symp_0(M,\omega)$.
Thus $\psi$ has flux $v$. It only remains to show that $\psi $ is a strongly unbounded
diffeomorphism of $(M,\omega)$.  Note that $\psi$ is not necessarily
strongly unbounded on $(M,\omega)$ since $\|\psi\|_U\geq \|\psi\|_M$.

By Lemma \ref{l:embeddingstrip}, we have a symplectic embedding of 
$(a,b)\times \mathbb{R}$ into $\tilde M$,  where $(a,b)$ is a small interval.
Recall that the symplectic diffeomorphisms $\psi$ is the one from the
proof of Proposition \ref{p:strongunonT}, except that now we consider it
on $(M,\omega)$. Thus on $(\tilde M,\tilde\omega)$
we have the same symplectic displacement as before. Hence the same arguments
of the proof of Proposition \ref{p:strongunonT} apply in this case. Thus
$\psi$  is strongly unbounded in $\symp_0(M,\omega).$
\end{proof}
\begin{rem}
From this result it follows that  $\textup{BI}_0(\Sigma_g,\omega)= \ham(\Sigma_g,\omega)$,
for $g\geq 1$.
The argument presented here is different from the proof that
appears in \cite{lalonde-polterovich}. But still the heart of our argument is
the same as their approach, namely
Proposition \ref{p:strongunonT}.
\end{rem}

For completeness we recall the following result that we will need later. It corresponds
to Lemma 4.2 from \cite{lalonde-pe}.

\begin{lem}
\label{l:fxgunbounded}
Consider $(M,\omega)$ and $(N,\eta)$  closed symplectic manifolds. If
$\psi\in\symp_0(M,\omega)$ is strongly unbounded, then $\psi\times \phi$
is unbounded for all $\phi\in\symp_0(N,\eta)$.
\end{lem}
\begin{proof}
Let $c$ be a positive real number. Since $\psi$ is strongly unbounded there
is  a Hamiltonian diffeomorphism $h$ of $(M,\omega)$ such that the lift
of $[\psi,h]$ to $\tilde M$ disjoins  a ball $B^{2n}(c_0)$ of capacity $c$. 
Note that $[\psi\times \phi,h\times 1_N]=[\psi,h]\times 1_N$
so the lift $[\psi,h]\tilde\ \  \times 1_N\ \ :\tilde M\times N\to \tilde M\times N$
disjoins $B^{2n}(c_0) \times N$. 

Thus by the stable version of the energy-capacity inequality of F. Lalonde and C. Pestieau
\cite{lalonde-pe}, we get
$$
c/2=c(B^{2n}(c_0))/2  \leq e(B^{2n}(c_0) \times N)\leq\| [\psi,h]\tilde\ \  \times 1_N  \| \leq  \| [\psi,h]  \times 1_N  \|.
$$
Therefore,
$$
\| [\psi\times \phi ,h\times 1_N ] \|=\| [\psi,h]\times 1_N \|\geq
c/2
$$
with $h\times 1_N$ a Hamiltonian diffeomorphism.
Hence $\psi\times \phi $ is unbounded.
\end{proof}

With this result at hand we can prove the following generalization
 of \cite[Theorem 1.3.C]{lalonde-polterovich} and  of
\cite[Lemma 4.3]{lalonde-pe}. 

\begin{thm}
\label{t:prod}
Let $(M,\omega)$ be a symplectic manifold that satisfies hypothesis (H), 
and let  $(N,\eta)$ be any closed symplectic manifold.
If $\psi\times\phi\in\symp_0(M\times N)$ is a bounded symplectic diffeomorphism,
then $\psi$ is a Hamiltonian diffeomorphism of $(M,\omega)$.
\end{thm}
\begin{proof}
Assume that $\psi$ is not a Hamiltonian diffeomorphism. By Theorem \ref{t:BIO}
we have $\textup{BI}_0(M,\omega)=\ham(M,\omega)$, so $\psi$ is an  unbounded symplectic
diffeomorphisms. Let $v\in H^1(M)/\Gamma_M$ be the flux of $\psi$.
Since $v$ is nonzero, it follows from  Proposition \ref{p:stromgunbM} that there
is $\psi_0\in\symp_0(M,\omega)$ that is strongly unbounded and has flux equal to $v$.

Therefore there  exists a Hamiltonian diffeomorphisms  $\alpha$ of $(M,\omega)$
such that $\psi=\psi_0\circ\alpha.$ Hence by   Lemma \ref{l:fxgunbounded}
we have that $\psi_0\times \phi$ is unbounded. Hence also
$\psi\times \phi$ is unbounded, which is a contradiction.
Therefore $\psi$ is a Hamiltonian diffeomorphism.
\end{proof}

\section{Proof of the main results}

The proof of Theorem \ref{t:BIO} follows from Proposition \ref{p:stromgunbM} where we showed
that for any $v$ in $H^1(M)/\Gamma_M$ there is an unbounded symplectic diffeomorphism
with flux $v$.

\begin{proof}[Proof of Corollary \ref{c:prodbounded}]
By the flux exact sequence we have  for any $\psi\in\symp_0(M\times N,\omega\oplus
\eta)$ there exist $\theta\in \ham(M\times N,\omega\oplus
\eta)$, $\psi_1\in \symp_0(M,\omega)$ and $\psi_2\in \symp_0(N,\eta)$ such
 that $\psi=\theta\circ (\psi_1\times \psi_2)$. Now if $\psi$ is bounded
it follows that $\psi_1\times \psi_2$ is also bounded. Now by
Theorem \ref{t:prod}, we have that each one of them is a Hamiltonian
diffeomorphism, hence $\psi_1\times \psi_2$ and $\psi$ are also Hamiltonian.
\end{proof}

\begin{proof}[Proof of Theorem \ref{t:prodflux}]
First note that since $(M,\omega)$ and $(N,\eta)$
satisfy {\em (H)},  by Corollary \ref{c:prodbounded} we have
$\textup{BI}_0(M\times N,\omega\oplus\eta)= \ham(M\times N,\omega\oplus\eta)$.

Consider $\psi\in \symp_0(M,\omega)$ and $\phi\in \symp_0(N,\eta)$ such that
$\psi\times \phi$ is a Hamiltonian diffeomorphism.
Thus  $\psi\times \phi$ is a bounded symplectic diffeomorphism. Hence from
Theorem \ref{t:prod} we have that $\psi$ and $\phi$ are Hamiltonian diffeomorphisms
as well. Therefore from Lemma \ref{l:prodHam}, we have that the flux group
of $(M\times N,\omega\oplus\eta)$ splits.
\end{proof}



\end{document}